\documentclass[11pt]{article}

\usepackage[a4paper,textwidth=16cm,textheight=24cm]{geometry}
\usepackage{tikz}
\usepackage{authblk}
\newtheorem{theorem}{Theorem}[section]
\newtheorem{proposition}[theorem]{Proposition}
\newtheorem{definition}[theorem]{Definition}

\newtheorem{lemma}[theorem]{Lemma}

\newtheorem{example}[theorem]{Example}
\newenvironment{proof}[1]{\textbf{#1} }{} %{\ \rule{0.5em}{0.5em}\newline}

\usepackage{amsmath,amssymb}
\newcommand{\Halmos}{$\Box$}
\newcommand{\leqc}{\leq_c}
\newcommand{\leqst}{\leq_{st}}

\newcommand{\Fbar}{\overline{F}}
\newcommand{\DOR}{{\rm DOR}}

\newcommand{\NWU}{{\rm NWU}}

\newcommand{\ITS}{{\rm InvSub}}

\DeclareMathOperator\Log{Log}
\DeclareMathOperator\Prob{P}
\DeclareMathOperator\Mexp{E}
\DeclareMathOperator\Var{Var}

\usepackage{natbib}
 \bibpunct[, ]{(}{)}{,}{a}{}{,}%

\title{Convex combinations of random variables stochastically dominate the parent for a new class of heavy tailed distributions}
\date{}

\author{Idir Arab}
\affil{CMUC, Dep. Mathematics, Univ. Coimbra, Portugal; idir.bhh@gmail.com}

\author{Tommaso Lando}
\affil{Department of Economics, University of Bergamo, Italy; tommaso.lando@unibg.it}

\author{Paulo Eduardo Oliveira}
\affil{CMUC, Dep. Mathematics, Univ. Coimbra, Portugal; paulo@mat.uc.pt}

\begin{document}
%%%%%%%%%%%%%%%%

\maketitle

\begin{abstract}
Stochastic dominance of a random variable by a convex combination of its independent copies has recently been shown to hold within the relatively narrow class of distributions with concave odds function, and later extended to broader families of distributions. A simple consequence of this surprising result is that the sample mean can be stochastically larger than the underlying random variable. We show that a key property for this stochastic dominance result to hold is the subadditivity of the cumulative distribution function of the reciprocal of the random variable of interest, referred to as the inverted distribution.
By studying relations and inclusions between the different classes for which the stochastic dominance was proved to hold, we show that our new class can significantly enlarge the applicability of the result, providing a relatively mild sufficient condition.

\smallskip

\noindent
\textbf{Keywords}: {Stochastic order, Inverted distribution,  Subadditivity, Odds function, Convex transform order}

\noindent
\textbf{AMS[2020] Classification}: {Primary 60E15; Secondary 91G70, 62P05}
\end{abstract}

%\HISTORY{Received: Month DD, YYYY; Accepted: Month DD, YYYY; Published Online: Month DD, YYYY}

%%%%%%%%%%%%%%%%%%%%%%%%%%%%%%%%%%%%%%%%%%%%%%%%%%%%%%%%%%%%%%%%%%%%%%

% Text of your paper here

\section{Introduction}\label{sec:intro}
Stochastic dominance is a widely-used tool in probability, which expresses some notion of one random variable being larger than another from a distributional point of view (see \cite{shaked2007}). The applications of this concept are numerous within different fields, such as statistics, economics, and finance, as is easily seen by the enormous references in the literature dealing with these concepts.
Although the topic has been studied extensively, recent results, discussed below, have outlined some ``surprising'' behaviours of stochastic dominance, especially when we consider sums of random variables.
While it may be intuitive, in a non-random setting, that summing the same quantity on both sides of some expression should not affect inequalities, and that a convex combination of points belongs to the convex hull, these basic principles are not generally true when random elements are involved.
For example, \cite{pomatto} have shown that, under some conditions, the ordering between a pair of random variables can be obtained by summing an independent ``noise'' to both, while \cite{superPareto2024} proved that, within a given family of probability distributions, a random variable can be dominated by convex combinations of independent copies from it. In both cases, the results are related to the variability, or the tail-heaviness, of the random variables involved.

In this paper, we show that the dominance result recently obtained by \cite{superPareto2024}, and then extended by \cite{ChenShneer2024} and \cite{Muller2024} to larger classes of distributions, can also hold under different, and in some cases broader, conditions.
To be more specific, we search for conditions under which, given $n$ i.i.d. copies of $X$, say $X_1,\ldots,X_n$, and weights $\theta_1,\ldots,\theta_n\geq0$ such that $\theta_1+\cdots+\theta_n=1$, we have
\begin{equation}
\label{eq:main-st}
X\leqst \theta_1 X_1+\cdots+\theta_n X_n.
\end{equation}
where $\leqst$ represents the standard stochastic dominance (see Definition~\ref{def:order1} below to recall the formal definition).
The implications of this result in terms of decision making under uncertainty, with meaningful applications in insurance and economic models, are quite remarkable, as it has been already explained by \cite{superPareto2024}, and complemented by applications to economical and management problems in \cite{ChenShneer2024}. Using the same terminology as in \cite{superPareto2024}, the relation in (\ref{eq:main-st}) represents an ``unexpected'' stochastic dominance result. Indeed, it is maybe intuitive to think that a convex combination is somewhere in between its components, which is actually the case for random variables with finite mean.
In particular, if $\Mexp X<\infty$, (\ref{eq:main-st}) holds trivially, with equality in distribution, if and only if exactly one coefficient is strictly positive. Differently, if $\Mexp X<\infty$ and at least two coefficients are strictly positive, (\ref{eq:main-st}) does not hold, as follows from the fact that the random variables $X$ and $\theta_1X_1+\cdots+\theta_nX_n$ have the same expectation and they are comparable in terms of variability, where $X$ is more variable that {$\theta_1X_1+\cdots+\theta_nX_n$} in terms of the convex order \citep{shaked2007}, as we will discuss later.
\cite{superPareto2024} proved that (\ref{eq:main-st}) holds if $X$ is an increasing convex transformation of a Pareto random variable with shape parameter 1, that these authors call a \textit{super-Pareto} random variable. This property, as described later, is equivalent to the concavity of the odds function (the class of distributions defined through shape properties of the odds functions has recently been studied in \cite{lando2023nonparametric}). However, the assumptions in \cite{superPareto2024} define a relatively narrow class of distributions, ruling out many important models for which stochastic dominance is still verified. Furthermore, the super-Pareto assumption implies that $X$ is absolutely continuous, while it is reasonable to expect that (\ref{eq:main-st}) may hold even for some discrete models, as it can be seen in some special cases.
The above examples justify the interest in finding weaker conditions for (\ref{eq:main-st}).
As a consequence, \cite{ChenShneer2024} introduced a new family of distributions, called \textit{super-heavy-tailed}, for which (\ref{eq:main-st}) is verified. This class, that we show to be closely related to the well-known \NWU\ family (see \cite{shaked2007}, for example), includes the super-Pareto, under some conditions about the convex transformation. Even more recently, and using different arguments, \cite{Muller2024} shows that (\ref{eq:main-st}) holds for convex transformations of Fr\'{e}chet and Cauchy variables, and denotes such classes as \textit{super-Fr\'{e}chet} and \textit{super-Cauchy}, respectively. In particular, he shows that the super-Cauchy class includes the super-Fr\'{e}chet, which includes the super-Pareto. However, the super-heavy-tailed family does not include the super-Cauchy, and vice-versa. We remark, moreover, that the super-Cauchy property is still not compatible with discrete models.

In our main result, we show that (\ref{eq:main-st}) holds under the subadditivity of the inverted distribution, that is, the cumulative distribution function of the reciprocal of $X$ (in the continuous case). This class, denoted as \ITS, includes the super-heavy-tailed class of \cite{ChenShneer2024}, and, under some conditions on the convex transformation, all the other classes discussed, except for the super-Cauchy family. On the other hand, we show, by examples, that the super-Cauchy class does not include our \ITS\ family. Finally, we note that, at the moment, the \ITS\ class seems the only one including some discrete distributions.

\section{Preliminaries and a new class of distributions}
\label{sec:main}
Let us start by introducing a few general concepts and terminology to be used in the sequel. Throughout this paper, ``increasing'' and ``decreasing'' are taken as ``non-decreasing'' and ``non-increasing'', respectively, and the generalised inverse of an increasing function $v$ is denoted as $v^{-1}(u)=\sup\{x\in\mathbb{R}:\,v(x)\leq u\}$. Moreover, a function $v$ is said to be subadditive if $v(x+y)\leq v(x)+v(y)$, for every $x,y$ in the domain of $v$. The function $v$ is called superadditive if the inequality is reversed. Finally, a function $v$ defined in $[0,+\infty)$ is said to be star-shaped if $v(0)=0$ and $\frac{v(x)}x$ is increasing. If $\frac{v(x)}x$ is decreasing, $v$ is called anti-star-shaped.

Given a random variable $X$, we shall represent by $F_X$ and $\Fbar_X=1-F_X$ its cumulative distribution and survival functions, possibly using other subscripts if different such objects are under consideration. Moreover, as we will be dealing with some discrete distributions, we represent by $F_X^-(x)=\Prob(X<x)$ the left-continuous cumulative distribution function of the random variable. In general, we shall not be assuming the existence of densities. We shall also be referring to the odds function $\Lambda_X(x)=\frac{F_X(x)}{\Fbar_X(x)}$, again possibly with different subscripts.
We recall the definition of stochastic dominance.
\begin{definition}
\label{def:order1}
Given two random variables $X$ and $Y$, we say that $Y$ stochastically dominates $X$, denoted as $X\leqst Y$, if $\Fbar_X(x)\leq\Fbar_Y(x)$, for every $x\in\mathbb{R}$.
\end{definition}
Note that we shall refer to the random variables or to their cumulative distribution functions, with the same notations, as is more convenient. In fact, the stochastic orders and the characterisations we will be discussing depend only on the cumulative distribution functions.

Bearing in mind that $F_{\frac1X}^-(x)=1-F_X(\frac1x)$ is generally referred to as the \textit{inverted distribution} of $X$, as this is the cumulative distribution function of $\frac1X$ in the continuous case, we introduce a new class of distributions that will be central to our main result.
\begin{definition}
\label{def-ITS-sub}
We say that a random variable $X$ such that $F_X(0)=0$ is \ITS\ (for ``inverted subadditive'') if $F_{\frac1X}^-$ is subadditive.
\end{definition}

We first present a simple characterisation of this class.
\begin{lemma}
\label{lem:ITS-carac}
A random variable $X$ is \ITS\ if and only if
\begin{equation}
\label{eq:sub-1x}
F_X(\tfrac{x}{\theta})+F_X(\tfrac{x}{1-\theta})\leq F_X(x)+1,\quad \forall x\geq0,\, \theta\in(0,1).
\end{equation}
\end{lemma}
\begin{proof}{Proof.}
The subadditivity of $1-F_X(\frac1x)$ is obviously equivalent to $1-F_X(\tfrac{x}{\theta})+1-F_X(\tfrac{x}{1-\theta})\geq 1-F_X(x)$, for every $x\geq0$ and $\theta\in(0,1)$, which is clearly a rewriting of (\ref{eq:sub-1x}). \Halmos
\end{proof}

\begin{example}
\label{rem:frechet}
A simple example of a class of distributions that are \ITS\ is obtained by considering random variables $X$ with Fr\'{e}chet distribution with shape parameter 1, that is, defined by the cumulative distribution function $\mathcal{H}(x)=e^{-1/x}$, for $x>0$.
In fact, given $\theta\in[0,1]$, $1+\mathcal{H}(x)-\mathcal{H}(\tfrac{x}{\theta})-\mathcal{H}(\tfrac{x}{1-\theta})=(1-e^{-\theta/x})(1-e^{-(1-\theta)/x})\geq 0$, so $\mathcal{H}$ satisfies (\ref{eq:sub-1x}).
Moreover, note that it is easily seen that $\Mexp X$ is infinite and the odds function is $\Lambda_{\mathcal{H}}(x)=\frac{e^{-1/x}}{1-e^{-1/x}}$, which is convex.
\end{example}

It is well-known that, for nonnegative functions, concavity implies subadditivity. Hence, when densities exist, a rather simple sufficient condition is available.
\begin{proposition}
\label{prop:ITS-dens}
Assume the nonnegative random variable $X$ has density $f_X$ such that its hazard rate $r_X(x)=\frac{f_X(x)}{\Fbar_X(x)}\leq\frac1x$, for $x>0$. Then $X$ is \ITS.
\end{proposition}
\begin{proof}{Proof.}
Note that the derivative of $\frac{1}{x}(1-F_X(\frac1x))$ has the same sign as $\frac{1}{x}f_X(\frac1x)+F_X(\frac1x)-1$, which, according to the assumption, is easily seen to be negative. It follows that $1-F_X(\frac1x)$ is anti-star-shaped, as it vanishes at zero, therefore, it is subadditive. \Halmos
\end{proof}

This very simple characterisation allows to show that a second family of distributions is \ITS.
\begin{example}
\label{ex:oddsFbar}
Consider nonnegative random variables $Y_b$, where $b\geq 0$, with survival function $\Fbar_b(x)=\frac{1}{1+x^b\Log(x+1)}$, for $x>0$. It is easily seen that these distributions satisfy the monotonicity assumption of Proposition~\ref{prop:ITS-dens} for $b\in(0,1)$ sufficiently small ($b\leq0.7$, although this is not an optimal bound). Therefore $Y_b$, for these suitably small values of $b$ are \ITS. Moreover, note that the odds function is $\Lambda_b(x)=x^b\Log(x+1)$, that can be checked to not be concave nor convex.
\end{example}

The existence of a density is not necessary. Indeed, the \ITS\ condition is also compatible with discrete models, as we illustrate in the next example.
\begin{example}
\label{ex:discITS}
Let $F_X(x)=(1-p)^{\lceil\frac1x\rceil}$, be defined in the completed half line $(0,+\infty]$, where $p\in(0,1)$ and $\lceil\cdot\rceil$ denotes the ceiling function. This cumulative distribution function is a right-continuous step function, with jumps at points $\frac 1k$, $k=1,2,3,\ldots,+\infty$, and in particular, it assigns positive mass $p$ to $+\infty$. This clearly implies that $\Mexp X=+\infty$. Now, $F_{\frac1X}^-(x)=1-F_X(\frac1x)=1-(1-p)^{\lceil x\rceil}$ is the left-continuous version of the geometric cumulative distribution function, which can be seen to be subadditive.
\end{example}

We close this section with a closure property about the class \ITS.
\begin{theorem}
\label{thm:inclusion}
Let $X$ be \ITS\ and $h$ a continuous star-shaped function. Then $h(X)$ is \ITS.
\end{theorem}
\begin{proof}{Proof.}
Note that, as $h$ is continuous, $F_{h(X)}(x)=F_X(h^{-1}(x))$. Since $h$ is star-shaped, its inverse, $h^{-1}$, is anti-star-shaped (see Lemma~4.1 in \cite{ineq-order2024}), hence, given $\theta\in(0,1)$, it follows that $h^{-1}(\tfrac{x}{\theta})\leq\tfrac{h^{-1}(x)}{\theta}$ and $h^{-1}(\tfrac{x}{1-\theta})\leq\tfrac{h^{-1}(x)}{1-\theta}$. As every function considered is increasing, we have
\begin{eqnarray*}
\lefteqn{F_{h(X)}(\tfrac{x}{\theta}) + F_{h(X)}(\tfrac{x}{1-\theta})} \\
 & = & F_X\left(h^{-1}(\tfrac{x}{\theta})\right) + F_X\left(h^{-1}(\tfrac{x}{1-\theta})\right) \\
 & \leq & F_X(h^{-1}(x))+1=F_{h(X)}(x)+1,
\end{eqnarray*}
using (\ref{eq:sub-1x}) for the last inequality, thus, taking into account Lemma~\ref{lem:ITS-carac}, the proof is concluded. \Halmos
\end{proof}

\section{Main result}
We present our main result stating that the \ITS\ property implies the stochastic dominance between a random variable and a convex combination of its independent copies.

\begin{theorem}
\label{thm:main}
If $X$ is \ITS, then the stochastic dominance (\ref{eq:main-st}) holds.
\end{theorem}
\begin{proof}{Proof.}
We proceed by induction on the number of random variables. Conditioning, using the independence of the random variables, and remembering they are nonnegative, it is easily seen that, for every $x\geq 0$ and $\theta\in(0,1)$,
$$
\Prob\left(\theta X_1+(1-\theta) X_2 > x\right) =1-\int_0^{x/\theta} F_X(\tfrac{x-\theta t}{1-\theta})\,F_X(dt).
$$
To find an upper bound for the integral, we consider the decomposition described in Figure~\ref{fig:f1}, from which follows easily that
$$
\int_0^{x/\theta} F_X(\tfrac{x-\theta t}{1-\theta})\,F_X(dt)
\leq F_X(\tfrac{x}{1-\theta})F_X(x)+F_X(x)\left(F_X(\tfrac{x}{\theta})-F_X(x)\right)
\leq F_X(x),
$$
using (\ref{eq:sub-1x}) for the last inequality.
\begin{figure}
\centering
\begin{tikzpicture}%[xscale=1.75,yscale=.8]
  \draw[->] (-.25,0) -- (7,0) coordinate[label = {below:$t$}] (xmax);
  \draw[->] (0,-.25) -- (0,4) coordinate; %[label = {left:$j$}] (ymax);
  \draw (6,-.1) coordinate[label = {below:{\small $\tfrac{x}\theta$}}] -- (6,.1);
  \draw (2.5,-.1) coordinate[label = {below:{\small $x$}}] -- (2.5,.1);
  \draw (-.1,3) coordinate[label = {left:{\small $F_X(\tfrac{x}{1-\theta})$}}] -- (.1,3);
  \draw (-.1,1.5) coordinate[label = {left:{\small $F_X(x)$}}] -- (.1,1.5);
  \draw (0,3) to [out=350,in=160] (2.5,1.5) to [out=340,in=180] (6,0);
  \draw[dashed] (0,3) -- (2.5,3) -- (2.5,0);
  \draw[dashed] (2.5,1.5) -- (6,1.5) -- (6,0);
\end{tikzpicture}
\caption{Upper bound for the integral in the initial induction step.}\label{fig:f1}
\end{figure}
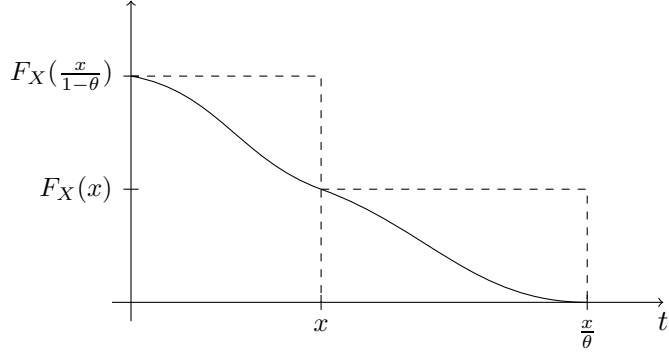
So, it follows that $\Prob\left(\theta X_1+(1-\theta) X_2 > x\right)\geq 1-F_X(x)=\Prob(X>x)$, so (\ref{eq:main-st}) holds for $n=2$.

Assume now that (\ref{eq:main-st}) holds whenever considering $n-1$ random variables. Given $\theta_1,\ldots,\theta_n>0$ satisfying $\theta_1+\cdots+\theta_n=1$, we have that
\begin{eqnarray*}
\lefteqn{\Prob\left(\theta_1 X_1+\cdots+\theta_n X_n > x\right)} \\
 & = & \Prob\left(X_n\geq\tfrac{x}{\theta_n}\right)+\Prob\left(\theta_1 X_1+\cdots+\theta_n X_n\geq x,X_n\leq\tfrac{x}{\theta_n} \right) \\
 & = & \Fbar_X(\tfrac{x}{\theta_n})+\int_0^{x/\theta_n}
     \Prob\left(\tfrac{\theta_1 X_1+\cdots+\theta_{n-1} X_{n-1}}{1-\theta_n}\geq\tfrac{x-\theta_n t}{1-\theta_n}\right)\,F_X(dt) \\
 & \geq & \Fbar_X(\tfrac{x}{\theta_n})+\int_0^{x/\theta_n} \Fbar_X(\tfrac{x-\theta_n t}{1-\theta_n})\,F_X(dt),
\end{eqnarray*}
using the induction hypothesis. We need now to find a lower bound for this integral, which may achieved using the decomposition depicted in Figure~\ref{fig:f2},
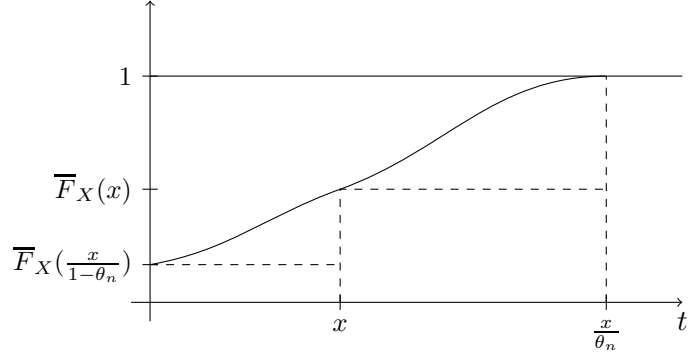
\begin{figure}
\centering
\begin{tikzpicture}%[xscale=1.75,yscale=.8]
  \draw[->] (-.25,0) -- (7,0) coordinate[label = {below:$t$}] (xmax);
  \draw[->] (0,-.25) -- (0,4) coordinate; %[label = {left:$j$}] (ymax);
  \draw (-.1,3) coordinate[label={left:{\small 1}}] -- (.1,3);
  \draw (0,3) -- (7,3);
  \draw (6,-.1) coordinate[label = {below:{\small $\tfrac{x}{\theta_n}$}}] -- (6,.1);
  \draw (2.5,-.1) coordinate[label = {below:{\small $x$}}] -- (2.5,.1);
  \draw (-.1,.5) coordinate[label = {left:{\small $\Fbar_X(\tfrac{x}{1-\theta_n})$}}] -- (.1,.5);
  \draw (-.1,1.5) coordinate[label = {left:{\small $\Fbar_X(x)$}}] -- (.1,1.5);
  \draw (0,.5) to [out=10,in=200] (2.5,1.5) to [out=20,in=180] (6,3);
  \draw[dashed] (0,.5) -- (2.5,.5);
  \draw[dashed] (2.5,0) -- (2.5,1.5) -- (6,1.5);
  \draw[dashed] (6,0) -- (6,3);
\end{tikzpicture}
\caption{Lower bound for the integral in the induction step.}\label{fig:f2}
\end{figure}
from which follows that
\begin{eqnarray*}
\lefteqn{\Prob\left(\theta_1 X_1+\cdots+\theta_n X_n > x\right)} \\
 & \geq & \Fbar_X(\tfrac{x}{\theta_n})+\int_0^{x/\theta_n} \Fbar_X(\tfrac{x-\theta_n t}{1-\theta_n})\,F_X(dt) \\
 & \geq & \Fbar_X(\tfrac{x}{\theta_n})
    +\Fbar_X(\tfrac{x}{1-\theta_n})F_X(x)
    + \Fbar_X(x)\left(F_X(\tfrac{x}{\theta_n})-F_X(x)\right) \\
 & \geq & \Fbar_X(\tfrac{x}{\theta_n})+\Fbar_X(\tfrac{x}{1-\theta_n})
      +\Fbar_X(x)\left(\Fbar_X(x)-\Fbar_X(\tfrac{x}{\theta_n})-\Fbar_X(\tfrac{x}{1-\theta_n})\!\right).
\end{eqnarray*}
Finally, noting that the subadditivity assumption implies that the large parenthesis is negative, it follows that  $\Prob\left(\theta_1 X_1+\cdots+\theta_n X_n > x\right)\geq \Fbar_X(\tfrac{x}{\theta_n})+\Fbar_X(\tfrac{x}{1-\theta_n}) \geq\Fbar_X(x)$, using (\ref{eq:sub-1x}) written in terms of the survival function, thus concluding the proof. \Halmos
\end{proof}

It is easy to see that, if $X$ has finite mean, then $X$ is more variable then $\theta_1X_1+\cdots +\theta_n X_n$ in terms of the convex order, that is, for every convex function $\phi$, $\Mexp\phi(\theta_1X_1+\cdots +\theta_n X_n)\leq \Mexp \phi(X)$ (this follows by using repeatedly the definition of convexity), meaning, for instance, that $X$ has larger variance (when it is finite) than the convex combination. Differently, as already remarked in \cite{superPareto2024}, the stochastic dominance stated in (\ref{eq:main-st}) is crucially linked to the fact that we are dealing with random variables with infinite means (see Proposition~2 in \cite{superPareto2024}). We present here another proof, using more elementary arguments.
\begin{proposition}
\label{prop:inf-mean}
Let $X_1,\ldots,X_n$ be independent random variables such that $\Prob(X_1=\cdots=X_n)<1$, and $\theta_1,\ldots,\theta_n>0$ such that $\theta_1+\cdots+\theta_n=1$. If (\ref{eq:main-st}) holds, then $X$ has infinite mean.
\end{proposition}
\begin{proof}{Proof.}
It is enough to prove the case $n=2$. Denote $Y=\theta X_1+(1-\theta)X_2$, where $X_1$ and $X_2$ are independent and have the same distribution as $X$ and are such that $\Prob(X_1\ne X_2)>0$. Moreover, assume $\Mexp(X)$ is finite. As then follows that $\Mexp(Y)=\Mexp(X)$, both finite, this, together with (\ref{eq:main-st}), implies that $X$ and $Y$ have the same distribution. Therefore, $\Var\left(\sqrt{X}\right)=\Var\left(\sqrt{Y}\right)$, implying that $\Mexp\left(\sqrt{X}\right)=\Mexp\left(\sqrt{Y}\right)$. But this is not possible, as Jensen's inequality implies that $\Mexp\left(\sqrt{Y}\right)=\Mexp\left(\sqrt{\theta X_1+(1-\theta)X_2}\right) > \theta\Mexp\left(\sqrt{X}\right)+(1-\theta)\Mexp\left(\sqrt{X}\right)=\Mexp\left(\sqrt{X}\right)$, the inequality being strict because $\Prob(X_1\ne X_2)>0$. \Halmos
\end{proof}

As a consequence of Theorem~\ref{thm:main} and Proposition~\ref{prop:inf-mean}, the following result is immediate.
\begin{proposition}
\label{prop:mean-fsub}
If $X$ is \ITS, then $\Mexp X$ is infinite.
\end{proposition}

\section{Comparing with earlier results}
We will now address the relations between the \ITS\ class and the other families of distributions implying (\ref{eq:main-st}), discussed in the Introduction.

We shall represent the cumulative distribution function of the Pareto with shape parameter 1 by $\mathcal{P}(x)=1-\frac1x$, for $x\geq1$. Note that it is straightforward to verify that $\mathcal{P}$ satisfies the subadditivity assumption in Theorem~\ref{thm:main}, that is, if $Z$ is such that $F_Z=\mathcal{P}$ then $Z$ is \ITS. Further, we shall denote the Cauchy distribution function by $\mathcal{C}(x)=\frac1\pi \arctan(x)+\frac12$, for $x\in\mathbb{R}$. If a random variable $T$ has cumulative distribution function $\mathcal{C}$ it can be verified that it is not in the class \ITS, however, its absolute value $\vert T\vert$ is \ITS.

\subsection{The super-Pareto class}
\label{subsec:superPareto}
This family is closely linked to the Pareto distributions, as follows from its definition.
\begin{definition}[\cite{superPareto2024}]
\label{def:superPareto}
A random variable $Y$ is super-Pareto if $Y\stackrel{d}{=}h(Z)$, where $F_Z=\mathcal{P}$ and $h$ is an increasing, convex and nonconstant function.
\end{definition}
As mentioned in \cite{superPareto2024}, the super-Pareto family includes the generalised Pareto distributions, the Burr distributions and the log-logistic distribution. Therefore, it includes some common models either in economical applications or in extreme value theory.
An adapted version of the main result in \cite{superPareto2024} is quoted below.
\begin{theorem}[adapted version of Theorem 1 in \cite{superPareto2024}]
\label{thm:main-velho}
The conclusion of Theorem~\ref{thm:main} holds if $X$ is super-Pareto.
\end{theorem}

Therefore, we are interested in relating the super-Pareto assumption with our \ITS\ assumption. We start by an equivalent characterisation of the super-Pareto family given below.
\begin{proposition}[Proposition 1 in \cite{superPareto2024}]
\label{prop:superP-conc}
A random variable $X$, with cumulative distribution function $F_X$, is super-Pareto if and only if $\frac{1}{\Fbar_X(x)}$ is concave.
\end{proposition}

Note that the concavity of $\Lambda_X(x)=\frac{F_X(x)}{\Fbar_X(x)}=\frac1{\Fbar_X(x)}-1$ implies the differentiability of $F_X$ at almost every point of the support, hence the concavity of $\Lambda_X$ is equivalent to the decreasingness of the odds rate $\lambda_X(x)=\Lambda_X^\prime(x)=\frac{f_X(x)}{\Fbar_X^2(x)}$, considering side derivatives at the points of nondifferentiability. This family of distributions has been addressed in \cite{lando2023nonparametric} or, more recently, in \cite{ineq-order2024}, being referred as the \DOR\ family (for ``\textit{decreasing odds rate}''), that can be characterised using an appropriate stochastic order.
We need some additional definitions to describe these relations more precisely.
\begin{definition}
\label{def:classes}
We say that a random variable $X$ is \DOR\ if its odds function $\Lambda_X(x)$ is concave.
\end{definition}
\begin{definition}
\label{def:order2}
Given two cumulative distribution functions $F_1$ and $F_2$, we say that, $F_1$ is smaller than $F_2$ in the convex transform order, represented by $F_1\leqc F_2$, if $F_2^{-1}\circ F_1$ is convex.
\end{definition}

The following result relates the shape of the odds function with the Pareto distribution through the convex transform order.
\begin{proposition}
\label{prop:conv}
$F_X$ is \DOR\ (or, equivalently, $X$ is super-Pareto) if and only if $\mathcal{P}\leqc F_X$. Analogously $F_X$ has increasing odds rate if and only if $F_X\leqc\mathcal{P}$.
\end{proposition}
\begin{proof}{Proof.}
Noting that the quantile of the Pareto is $\mathcal{P}^{-1}(u)=\frac1{1-u}$, we have $\mathcal{P}^{-1}\circ F_X=\frac1{\Fbar_X}=\Lambda_X+1$, so the result follows immediately. \Halmos
\end{proof}

As mentioned in Example~\ref{rem:frechet}, the Fr\'{e}chet class, with cumulative distribution functions $\mathcal{H}(x)=e^{-1/x}$ satisfies the assumption on Theorem~\ref{thm:main}. On other hand, its odds functions $\Lambda_{\mathcal{H}}$ is convex, hence $\mathcal{H}\leqc\mathcal{P}$, so random variables with cumulative distribution function $\mathcal{H}$ are not super-Pareto. This provides an example where our main result Theorem~\ref{thm:main} implies (\ref{eq:main-st}), while Theorem~\ref{thm:main-velho} is not applicable, as its assumptions are not satisfied. Further, by transitivity, note that if $X$ is super-Pareto we have $\mathcal{H}\leqc\mathcal{P}\leqc F_X$.

We now present one example showing that shape conditions about the odds function are not the most appropriate way to find sufficient conditions for the stochastic dominance (\ref{eq:main-st}).
\begin{example}
\label{ex:oddsFbar-c}
The random variables $Y_b$ introduced in Example~\ref{ex:oddsFbar} have odds function $\Lambda_b(x)=x^b\Log(x+1)$ that, as mentioned before, are not concave nor convex for $b<1$. Nevertheless, as referred in Example~\ref{ex:oddsFbar}, $Y_b$, for $b\in(0,1)$ suitably small, is \ITS.
\end{example}

Finally, we relate the super-Pareto class with the \ITS\ family.
\begin{theorem}
\label{cor:inclusion}
If $X$ is nonnegative super-Pareto then $X$ is \ITS.
\end{theorem}
\begin{proof}{Proof.}
According to Proposition~\ref{prop:superP-conc}, $X$ is nonnegative \DOR. Let $Z$ be such that $F_Z=\mathcal{P}$ and $Y=Z-1$, where $F_Y(x)=\frac{x}{x+1}$, for $x\geq0$. Given the shift invariance of the convex transform order, $X$ being super-Pareto is equivalent to $X\stackrel{d}{=}h(Y)$, where $h$ is increasing convex and $h(0)=0$, hence $h$ is star-shaped. It is straightforward to verify that $Y$ is \ITS, so the conclusion follows by applying Theorem~\ref{thm:inclusion}. \Halmos
\end{proof}

\subsection{The super-heavy-tailed class}
\label{subsec:superheavy}
In \cite{ChenShneer2024} the class of \textit{super-heavy-tailed} distributions is introduced. As we shall see, this is closely related to a family of distributions that is well established in the literature.
\begin{definition}
\label{def:superheavy}
We say that a random variable $X$ such that $F_X(0)=0$ is
    \begin{enumerate}
    \item
    (\textbf{definition~2 in \cite{ChenShneer2024}}) super-heavy-tailed if $-\Log F_X(\frac1x)$ is subadditive;
    \item
    \NWU\ if $\Fbar_X(x)\Fbar_X(y)\leq\Fbar_X(x+y)$, for every $x$ and $y$.
    \end{enumerate}
\end{definition}
The characterisation of $X$ being \NWU, a well-known class of distributions, may be rewritten as $\Log\Fbar_X$ is superadditive (see the initial notes in \cite{shaked2007}), together with being nonnegative. Note that the nonnegativity is not strictly necessary for the above definitions.

It is readily seen that in the continuous case, $X$ being super heavy-tailed is equivalent to $\frac1X$ being \NWU. Note that, in the general case, $F_X(\frac1x)=1-F_{\frac1X}^-(x)$. 

The following characterisation is straightforward.
\begin{lemma}
\label{lem:inv-NWU}
A random variable $X$ is super-heavy-tailed if and only if
\begin{equation}
\label{eq:inv-NWU}
F_X(\tfrac{x}{\theta})F_X(\tfrac{x}{1-\theta})\leq F_X(x),\quad\forall x\geq0,\,\theta\in(0,1).
\end{equation}
\end{lemma}

We may now quote the stochastic dominance proved in \cite{ChenShneer2024}.
\begin{theorem}[adapted version of Theorem~1 in \cite{ChenShneer2024}]
\label{thn:ChenShneer}
The conclusion of Theorem~\ref{thm:main} holds if $X$ is super-heavy-tailed.
\end{theorem}

We now discuss the relation between the super-heavy-tailed class and the \ITS\ we introduced.
\begin{theorem}
\label{thm:inv-nwu-fsub}
If $X$ is super-heavy-tailed, then $X$ is \ITS.
\end{theorem}
\begin{proof}{Proof.}
As noted after Definition~\ref{def:superheavy}, $X$ being super-heavy-tailed means that $\Log(1-F_{\frac1X}^-(x))=\Log F_X(\frac1x)$ is superadditive, that is $\Log F_X(\frac{1}{x+y})\geq\Log F_X(\frac1x)+\Log F_X(\frac1y)$, which implies that $$
F_X\left(\tfrac{1}{x+y}\right)\geq\exp\left(\Log F_X\left(\tfrac1x\right)+\Log F_X\left(\tfrac1y\right)\right)
  \geq F_X\left(\tfrac1x\right)+F_X\left(\tfrac1y\right)-1,
$$
as the exponential is superadditive. Finally, going to the complementary sets, this last inequality is just $\Fbar_X(\frac{1}{x+y})\leq\Fbar_X(\frac1x)+\Fbar_X(\frac1y)$, that is, $\Fbar_X(\frac1x)$ is subadditive, that is, $X$ is \ITS. \Halmos
\end{proof}

Of course, this result implies that Theorem~1 in \cite{ChenShneer2024} is contained in Theorem~\ref{thm:main}. Moreover, the following example shows that the \ITS\ class is strictly larger than the super-heavy-tailed family.

\begin{example}
\label{ex:strange}
Consider a random variable $R_a$ with cumulative distribution functions of the form $a^{-1/x}\left(1+\frac1{e^{2x}-1}\right)$. Choosing $a\in(0,1)$ sufficiently small (say, for example, $a=0.5$), it can be verified that (\ref{eq:sub-1x}) is satisfied, by checking the assumption in Proposition~\ref{prop:ITS-dens}, so $R_a$ is \ITS. Differently, (\ref{eq:inv-NWU}) is not, so $R_a$ is not super-heavy-tailed. 
\end{example}

\subsection{The super-Cauchy class}
\label{subsec:superCauchy}
In \cite{Muller2024}, in a similar way as to the ``\textit{super-something}'' classes already discussed, the following class is considered.
\begin{definition}
\label{def:superCauchy}
A random variable $X$ is said super-Cauchy if $\mathcal{C}\leq_c F_X$.
\end{definition}
Note that, unlike the previously defined classes, this is the only one that allows the support to be the whole $\mathbb{R}$. 
Therefore, no inclusion relationship is to be expected between the super-Cauchy and our \ITS\ class, unless, or course, we introduce some boundedness assumption about the support. %restrict ourselves to nonnegative random variables.

The following result is proved in \cite{Muller2024}.
\begin{theorem}[Theorem~2.10 in \cite{Muller2024}]
\label{thm:muller}
$X$ is super-Pareto $\quad\Rightarrow\quad$ $X$ is super-Fr\'{e}chet $\quad\Rightarrow\quad$ $X$ is super-Cauchy. Moreover, the conclusion of Theorem~\ref{thm:main} holds if $X$ is super-Cauchy.
\end{theorem}

Theorem~2.14 in \cite{Muller2024} shows that, if $T$ has cumulative distribution function $\mathcal{C}$, then $\vert T\vert$ is super-Cauchy while $T$ is not super-heavy-tailed. Nevertheless, $\vert T\vert$ can easily be seen to be \ITS. Moreover, to show that the super-Cauchy class does not contain the \ITS\ one, we provide the following example.

\begin{example}
\label{ex:non-superCauchy}
Take the cumulative distribution function $V(x)=e^{-1/\sqrt{x}}\left(1+e^{2x-1}\right)$, for $x>0$. It can be seen that $1-V(\frac1x)$ is anti-star-shaped, so $V$ is \ITS. Moreover, it is easily verified that $\mathcal{C}^{-1}\circ V$ is not concave nor convex, so this distribution is not super-Cauchy. 
\end{example}

\section{Characterisations based on a new inverted-subadditive order}

Many relevant nonparametric families of distributions, such as the IHR, DHR, NBU, and NWU classes, discussed in the book of \cite{marshall2007}, and many others, can be defined via a suitable stochastic order with respect to some reference distribution. This is also the case, for example, of the super-Pareto, super-Fr\'{e}chet, and super-Cauchy families, which are defined through the convex transform order, and using the aforementioned distributions as reference points. Such types of characterisations are very useful, since they allow to determine a benchmark distribution and to establish relations beween classes using transitivity and other properties of stochastic orders. We now show that also the super-heavy tailed and the \ITS\ classes can also be defined again using the Fr\'{e}chet and the Pareto distributions, respectively, as a benchmark. Nevertheless, in this case the order is of different nature, and it is introduced in the following definition.
\begin{definition}
\label{def:InvSub-order}
Given two cumulative distribution functions $F_1$ and $F_2$ we say that $F_1$ is smaller than $F_2$ in inverted-subadditive order, represented by $F_1\leq_{i-sb}F_2$, if $\frac{1}{F_2^{-1}(F_1^-(\frac1x))}$ is subadditive.
\end{definition}
Note that the inverted-subadditive order is related to inverted distributions, taking into account that, in the continuous case, the quantile function of the inverted distribution of $F_2$ is $\frac{1}{F_2^{-1}(1-p)}$. Moreover, differently from other transform orders, such as the convex one, this new order is not shift-invariant. In fact, shifts may be crucial in this case, since we are dealing with inverted distributions. For the same reason, it is also clear that the inverted-subadditive order rules out distributions whose supports strictly include the origin.

The following result gives a complete description of the super-heavy-tailed class, avoiding the strict continuity difficulties.
\begin{proposition}
\label{cor:invNWU}
A random variable $X$ is super-heavy-tailed if and only if $F_X\leq_{i-sb}\mathcal{H}$.
\end{proposition}
\begin{proof}{Proof.}
Note that $\mathcal{H}^{-1}(p)=-\frac1{\Log p}$, so $\frac{1}{\mathcal{H}^{-1}(F_X(\frac1x))}=-\Log(F_X(\frac1x))$.
\Halmos
\end{proof}

We may obtain a similar characterisation for the \ITS\ class using the subadditive order and choosing the appropriate benchmark distribution. %The following is straightforward by computing the quantile function of the Pareto.

\begin{proposition}
\label{cor:fsub}
A random variable $X$ such that $F_X(0)=0$ is \ITS\ if and only if\linebreak $F_X\leq_{i-sb}\mathcal{P}$.
\end{proposition}
\begin{proof}{Proof.}
As $\mathcal{P}^{-1}(p)=\frac{1}{1-p}$, we are reduced to verifying the subadditivity of $1-F_X(\frac1x)$. \Halmos
\end{proof}

Note that, by transitivity of the inverted-subadditive order, Propositions~\ref{cor:invNWU} and \ref{cor:fsub}, together with $\mathcal{H}\leq_{i-sb}\mathcal{P}$, provide an alternative proof of Theorem~\ref{thm:inv-nwu-fsub}.

It is interesting to note that choosing the Fr\'{e}chet as a benchmark, compared to the Pareto, yields a larger class when we use the convex transform order, and a smaller one if we use the inverted-subadditive order, due to the different nature of these orders.

\section*{Acknowledgement}
T.L. was supported by the Italian funds ex MURST 60\% 2022. I.A. and P.E.O. were partially supported by the Centre for Mathematics of the University of Coimbra UID/MAT/00324/2020, funded by the Portuguese Government through FCT/MCTES and co-funded by the European Regional Development Fund through the Partnership Agreement PT2020.
P.E.O. expresses his gratitude to the University of Central Lancashire, Cyprus, for receiving him during a large period of development of this work.

\bibliographystyle{elsarticle-harv} % outcomment this and next line in Case 1
\bibliography{biblio-peo} % if more than one, comma separated

%%%%%%%%%%%%%%%%%
\end{document}